\documentclass[11pt,a4paper]{amsart}
\usepackage[pdftex]{graphicx}
\usepackage[pdftex]{xcolor}

\usepackage{mathtools}

\makeatletter
\newtheorem*{rep@theorem}{\rep@title}
\newcommand{\newreptheorem}[2]{%
\newenvironment{rep#1}[1]{%
 \def\rep@title{#2 \ref{##1}}%
 \begin{rep@theorem}}%
 {\end{rep@theorem}}}
\makeatother

\definecolor{RedOrange}{cmyk}{ 0, 0.77, 0.87, 0}
\definecolor{RoyalPurple}{cmyk}{ 0.84, 0.53, 0, 0}
\definecolor{YellowGreen}{cmyk}{ 0.44, 0, 0.74, 0}
\definecolor{Fuchsia}{cmyk}{ 0.47, 0.91, 0, 0.08}
\definecolor{Blue}{cmyk}{ 0.84, 0.53, 0, 0}
\definecolor{BlueViolet}{cmyk}{ 0.84, 0.53, 0, 0}
\definecolor{Black}{cmyk}{ 0.75, 0.68, 0.67, 0.9}
\usepackage{verbatim}
\usepackage{amsmath}
\usepackage{amssymb, mathrsfs}
\usepackage{amsbsy}

\usepackage{amscd}
\usepackage{amsthm}

\usepackage[english]{babel}
\usepackage{todonotes}
\usepackage[T1]{fontenc}

\usepackage{color}

\newcommand{\R}{\mathbb{R}}
\newcommand{\B}{\mathbb{B}}

\newcommand{\N}{\mathbb{N}}
\newcommand{\e}{\varepsilon}

\newcommand{\E}{\mathbb{E}}
\newcommand{\Z}{\mathbb{Z}}

\renewcommand{\P}{\mathbb{P}}

\newcommand{\rmL}{\mathrm{L}}

\newcommand{\lin}{\left[\kern-0.15em\left[}
\newcommand{\rin} {\right]\kern-0.15em\right]}
\newcommand{\linf}{[\kern-0.15em [}
\newcommand{\rinf} {]\kern-0.15em ]}
\newcommand{\ilin}{\left]\kern-0.15em\left]}
\newcommand{\irin} {\right[\kern-0.15em\right[}

\usepackage{constants}

\newconstantfamily{c}{symbol=c}

\newconstantfamily{a}{symbol=\alpha}

\newcommand{\secno}[1]{\thesection.\arabic{#1}}
\newconstantfamily{kE}{
symbol=\mathcal{E},
format=\secno,
reset={section}
}

\renewcommand{\tilde}{\widetilde}

\newtheorem{lem}{Lemma}[section]
\newtheorem{remark}[lem]{Remark}
\newtheorem{prop}[lem]{Proposition}
\newtheorem{thm}[lem]{Theorem}

\newtheorem{cor}[lem]{Corollary}

\newtheorem {Def}[lem] {Definition}

\usepackage{color}
\definecolor{lilas}{RGB}{182, 102, 210}

\numberwithin{equation}{section}

\title[Divergence of non-random fluctuation in FPP]
{Divergence of non-random fluctuation in First-passage percolation}
\date{\today}
\author{Shuta Nakajima} 
\address[Shuta Nakajima]
{Graduate School of Mathematics, University Nagoya.}
\email{njima@math.nagoya-u.ac.jp}


\begin{document}
\maketitle

\begin{abstract}
We study the non-random fluctuation in first passage percolation and show that it diverges. We also prove the divergence of non-random shape fluctuation, which was predicted in [Yu Zhang. The divergence of fluctuations for shape in first passage percolation. {\em Probab. Theory. Related. Fields.} 136(2) 298--320, 2006].
\end{abstract}

\section{Introduction}
First-passage percolation is a dynamical model of infection, which was introduced by Hammersley and Welsh \cite{HW65}. The model has received much interests both in mathematics and physics because it has rich structures from the viewpoint of the random metric and it is related to the KPZ-theory \cite{KS91}. See \cite{ADH} on the backgrounds and related topics.\\

We consider the First-passage percolation (FPP) on the lattice $\Z^d$ with $d\geq{}2$. The model is defined as follows. The vertices are the elements of $\Z^d$. Let us denote by $\mathbb{E}^d$  the set of edges:
$$\mathbb{E}^d=\{\{ v,w \}|~v,w\in\Z^d,~|v-w|_1=1\},$$
where we set $|v-w|_1=\sum^d_{i=1}|v_i-w_i|$ for $v=(v_1,\cdots,v_d)$, $w=(w_1,\cdots,w_d)$. Note that we consider non-oriented edges in this paper, i.e., $\{ v,w \}=\{ w,v \}$ and we sometimes regard $\{ v,w \}$ as a subset of $\Z^d$ with a slight abuse of notation. We assign a non-negative random variable $\tau_e$ to each edge $e\in \mathbb{E}^d$, called the passage time of the edge $e$. The collection $\tau=\{\tau_e\}_{e\in \mathbb{E}^d}$ is assumed to be independent and identically distributed with common distribution $F$. \\
    
  A path $\gamma$ is a finite sequence of vertices $(x_1,\cdots,x_l)\subset\Z^d$ such that for any $i\in\{1,\cdots,l-1\}$, $\{x_i,x_{i+1}\}\in \mathbb{E}^d$. Given an edge $e\in \mathbb{E}^d$, we write $e\in \gamma$ if there exists $i\in \{1\cdots,l-1\}$ such that $e=\{x_{i},x_{i+1}\}$.  Given a path $\gamma$, we define the passage time of $\gamma$ as
$${\rm T}(\gamma)=\sum_{e\in\gamma}\tau_e.$$
For $x\in\R^d$, we set $[x]=([x_1],\cdots,[x_d])$  where $[a]$ is the greatest integer less than or equal to $a\in\R$. Given two vertices $v,w\in\R^d$, we define the {\em first passage time} between $v$ and $w$ as
$${\rm T}(v,w)=\inf_{\gamma:[v]\to [w]}{\rm T}(\gamma),$$
where the infimum is taken over all finite paths $\gamma$ starting at $[v]$ and ending at $[w]$. A path $\gamma$ from $v$ to $w$ is said to be {\em optimal} if it attains the first passage time, i.e., ${\rm T}(\gamma)={\rm T}(v,w)$. We define $G(t)=\{x\in\R^d|~\E {\rm T}(0,x)\leq t\}$.\\

By Kingman's subadditive ergodic theorem \cite{King68}, if $\E \tau_e<\infty$, then for any $x\in\R^d$, there exists a non-random constant ${\rm g}(x)\ge 0$ such that

\begin{equation}\label{kingman}
  {\rm g}(x)=\lim_{t\to\infty}t^{-1} {\rm T}(0,t x)=\lim_{t\to\infty}t^{-1} \E[{\rm T}(0,t x)]\hspace{4mm}a.s.
\end{equation}
This ${\rm g}(x)$ is called the {\em time constant}. Note that, by subadditivity, if $x\in\Z^d$, then ${\rm g}(x)\le \E {\rm T}(0,x)$ and moreover for any $x\in\R^d$, ${\rm g}(x)\le \E {\rm T}(0,x)+2d\E \tau_e$. It is easy to check the homogeneity and the convexity: ${\rm g}(\lambda x)=\lambda {\rm g}(x)$ and ${\rm g}(r x+(1-r)y)\le r {\rm g}(x)+(1-r){\rm g}(y)$ for $\lambda\in\R$, $r\in[0,1]$ and $x,y\in\R^d$. It is well-known that if $F(0)<p_c(d)$, then ${\rm g}(x)>0$ for any $x\neq 0$, see, e.g., \cite{Kes86}. Therefore, if $F(0)<p_c(d)$, then $g:\R^d\to \R_{\geq 0}$ is a norm. We use ${\rm g}(x)\leq 2d\E \tau_e|x|$ and $\E {\rm T}(0,x)\leq 2d\E \tau_e|x|$ for $x\in\R^d$ with $|x|\ge 1$ many times in the proof without any comments.\\
  
\subsection{Backgrounds and related topics}
   Hammersley and Welsh \cite{HW65} have proved that $\frac{1}{N}{\rm T}(0,N\mathbf{e}_1)$ converges ${\rm g}(\mathbf{e}_1)$ in probability when $d=2$. This statement was strengthened by Kingman \cite{King68} as stated in \eqref{kingman}. Since then, the rate of this convergence becomes one of the most important problems in this model. The difference ${\rm T}(0,x)-{\rm g}(x)$ can be naturally divided into the {\em random fluctuation} part and the {\em non-random fluctuation} part as follows:
  $${\rm T}(0,x)-{\rm g}(x)=\underbrace {{\rm T}(0,x)-\E {\rm T}(0,x)}_{\text{random}}+ \underbrace{\E {\rm T}(0,x)-{\rm g}(x)}_{\text{non-random}}.$$
  
   Let us briefly review the earlier work. It is widely believed that there exist universal constants $\chi(d),\chi'(d) \ge 0$ such that, as $|x|\to \infty$,
  \begin{equation}\label{exponent}
    {\rm T}(0,x)-\E {\rm T}(0,x)\sim \sqrt{\text{Var}({\rm T}(0,x))}\sim |x|^{\chi(d)}\text{ and }\E {\rm T}(0,x)-{\rm g}(x)\sim |x|^{\chi'(d)}
    \end{equation}in a suitable sense. The term ``universal'' means that these values are independent of distribution of $\tau$. To state the previous work precisely, we introduce four relevant quantities:
  $$\bar{\chi}(d)=\lim_{t\to\infty}\sup_{|x|\geq t}\frac{\log{\text{Var}({\rm T}(0,x))}}{2\log{|x|}},~\underline{\chi}(d)=\lim_{t\to\infty}\inf_{|x|\geq t}\frac{\log{\text{Var}({\rm T}(0,tx))}}{2\log{t}},$$
  \begin{equation}\label{NF-exponent}
    \bar{\chi}'(d)=\lim_{t\to\infty}\sup_{|x|\geq t}\frac{\log{|\E {\rm T}(0,x)-{\rm g}([x])|}}{\log{|x|}},~\underline{\chi}'(d)=\lim_{t\to\infty}\inf_{|x|\geq t}\frac{\log{|\E {\rm T}(0,x)-{\rm g}([x])|}}{\log{|x|}}.
    \end{equation}
  Due to the work of Kesten~\cite{Kes93}, it is (the best currently) known that $0\leq \underline{\chi}(d)\leq \bar{\chi}(d)\leq 1/2$  under the condition that the second moment of $\tau$ is finite. On the other hand, Newman and Piza showed that $\bar{\chi}(2)\geq 1/8$ for a useful distributions under an exponential moment condition \cite{NP95}, where useful distributions are defined in \eqref{Def:useful} below.\\

  Let us move on to the previous researches on the non-random fluctuation. Alexander \cite{Alex97} found the relationship between $\bar{\chi}(d)$ and $\bar{\chi}'(d)$ and he proved $\bar{\chi}'(d)\leq 1/2$ with an exponential moment condition, which was later relaxed to a low moment condition  in \cite{DK16}. For the lower bounds, it is proved that $ \underline{\chi}'(d)\ge -1$ in \cite{Kes93} and $\bar{\chi}'(d)\ge -1/2$ in \cite{ADH15} with an exponential moment condition.\\
  
  remarkarkably, it was shown in \cite{ADH15} that $\chi(d)$ and $\chi'(d)$ in \eqref{exponent} are actually the same under an assumption of the existence of $\chi(d)$ in a suitable sense. In fact, it is expected that they have the exactly same growth \cite{DW16,Joha00}. As a consequence, the above four quantities should be all the same, which are called the {\em fluctuation exponent} collectively. From the KPZ-theory, it is conjectured that $\chi(2)(=\chi'(2))=1/3$. However for higher dimensions, the values are unknown. Some physicists predicted that in sufficiently high dimensions, $\chi(d)=0$ \cite{CD90,HH89,NR88}. If it is correct, the further problem can be conceivable  whether the random fluctuation and the non-random fluctuation diverge or not. In this paper, we prove that the latter diverges for any dimension $d\ge 2$, which is the first result around related models. Accordingly, we believe that the former diverges too.\\

\subsection{Main results}
   We restrict our attention to the following class of distributions. A distribution $F$ is said to be {\em useful} if  
   \begin{equation}\label{Def:useful}
     \P(\tau_e=F^-)<
    \begin{cases}
    p_c(d) & \text{if $F^-=0$} \\
    \vec{p}_c(d)& \text{otherwise},
    \end{cases}
    \end{equation}
   where $p_c(d)$ and $\vec{p}_c(d)$ stand for the critical probabilities for $d$-dimensional percolation and oriented percolation model, respectively and $F^-$ is the infimum of the support of $F$. Note that if $F$ is continuous, i.e., $\P(\tau_e=a)=0$ for any $a\in \R$, then $F$ is useful.\\

   Let us define an euclidean ball ${\rm B}(x,r)$ for $x\in\R^d$ and $r>0$ as
   $${\rm B}(x,r)=\{y\in\R^d|~d(x,y)\le r\}.$$
   We write a unit ball $\B_d$ with respect to the norm $g$ as
   $$\B_d=\{x\in\R^d|~{\rm g}(x)\leq 1\}.$$

   \begin{Def}
   A point $\mathbf{x}_d\in \partial \B_d$ is said to be {\it directional flat} if there exist $x_1\in\R^d$ and $r>0$ such that ${\rm B}(x_1,r)\subset \B_d$ and $\mathbf{x}_d\in \partial {\rm B}(x_1,r)$.
   \end{Def}
   \begin{thm}\label{thm-main2}
   
     Suppose that $F$ is useful and $\E [\tau_e^2 (\log{\tau_e})_+] <\infty$. Let $\mathbf{x}_d\in \partial \B_d$ be a directional flat point.  Then there exist a sequence $(x_n)_{n\in \N}\subset \Z^d$ and $c>0$ such that $|x_n|_1=n$, $x_n/|x_n|\to\mathbf{x}_d/ |\mathbf{x}_d|$ and for any sufficiently large $n\in\N$ 
     \begin{equation}
       |\E {\rm T}(0,x_n) -{\rm g}(x_n)|\geq c(\log{\log{n}})^{1/d}.
       \end{equation}
     In particular, by Jensen's inequality,     
     \begin{equation}
       \lim_{n\to\infty} \E |{\rm T}(0,x_n) -{\rm g}(x_n)|=\infty.\label{fluc-time}
       \end{equation}

   \end{thm}
   The moment condition $\E [\tau_e^2 (\log{\tau_e})_+] <\infty$ above will be used to get the sublinear variance of the first passage time (see Lemma~\ref{sublinea-lem}). \eqref{fluc-time} means that the fluctuation of the first passage time around the time constant diverges. It may suggest that the fluctuation of the first passage time around the mean also diverges.
\begin{figure}[b]
  \scalebox{0.8}{\includegraphics[width=8.0cm]{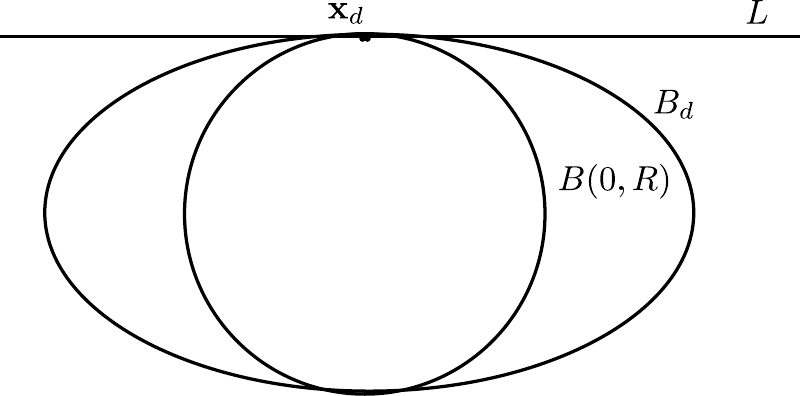}}
  \hspace{5mm}
  \scalebox{0.8}{\includegraphics[width=8.0cm]{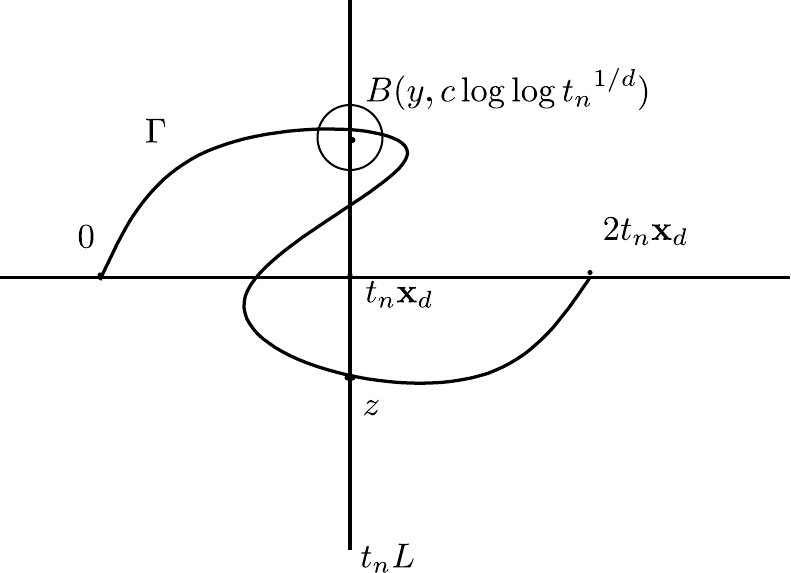}}
  \caption{}
  \label{fig:NF}
Left:  Figure of $\mathbf{x}_d$ and $L$.\\
Right: The schematic picture of Step 2 in the proof of Lemma~\ref{key:lem}.
\end{figure}
   \begin{remark}
     There certainly exists a directional flat point. In fact, we can take an arbitrary point $\mathbf{x}_{d}\in \partial B_d \cap \partial {\rm B}(0,R)$, where  $R=\sup\{r>0|~{\rm B}(0,r)\subset \B_d\}$ (see Figure~\ref{fig:NF}).
   \end{remark}
    
   We also consider the fluctuation of $G(t)$ from $t\B_d$.
\begin{Def}
   For $l>0$ and a subset $\Gamma$ of $\R^d$, let
   $$\Gamma^-_l=\{v\in\Gamma|~d(v,\Gamma^c)\geq{}l\}\text{ and }\Gamma^+_l=\{v\in\R^d|~d(v,\Gamma)\leq{}l\},$$
   where $d$ is the Euclid distance. Given three sets $A,B,C\subset \R^d$, we define the fluctuation of $A$ from $B$ inside $C$ as
    $${\rm F}_C(A,B)=\inf\{\delta>0|~B^-_{\delta}\cap C\subset A\cap C \subset B^+_{\delta}\cap C\}.$$
\end{Def}

\begin{remark}
  The results in this paper will
be formulated by using ${\rm F}_C(A, B)$ but they can also be proved for $d_H(A \cap C, B \cap  C)$ by essentially the same arguments, where $d_H$ is the Hausdorff distance.
  \end{remark}

  To consider the directional fluctuation, we define the following cone. 
  \begin{Def}
    Given $\theta\in \R^d$ and $r>0$, let
    $$\rmL(\theta,r)=\{a\cdot \mathbf{v}|~a\in[0,\infty),~\mathbf{v}\in {\rm B}(\theta,r)\}.$$
  \end{Def}
  Let us consider the divergence of non-random shape fluctuation $F(G(t),t\B_d)$, which was predicted in Remark~2 of \cite{Zhang06}. 
\begin{cor}\label{thm-main}
   Suppose that $F$ is useful and $\E [\tau_e^2 (\log{\tau_e})_+] <\infty$. Let $\mathbf{x}_d\in \partial \B_d$ be a directional flat point. Then for any $r>0$, there exists $c>0$ such that for any sufficiently large $t$,
  $${\rm F}_{\rmL(\mathbf{x }_d,r)}(G(t),t\B_d)\geq c (\log{\log{t}})^{1/d}. $$
\end{cor}
The next theorem shows that $\underline{\chi}'(d)$ defined in \eqref{NF-exponent} is non-negative.
\begin{thm}\label{lem-positive-NF}
  Suppose that $F$ is non-degenerate and $\E \tau_e<\infty$. Then, there exists $c>0$ such that for any $x\in \Z^d\backslash\{0\}$,
  \begin{equation}\label{positive-NF}
    \E {\rm T}(0,x)-{\rm g}(x)\geq c.
  \end{equation}
    
  \end{thm}
\subsection{Notation and terminology}
This subsection collects some notations and terminologies for the proof.
\begin{itemize}

\item We denote the Euclidean distance between two sets as 
  $$d(A,B)=\inf\{d(x,y)|~x\in A,~y\in B\}\text{\hspace{4mm}for $A,B\subset \R^d$}.$$
  When $A=\{x\}$, we write $d(x,B)$.
  
\item Let $F^-$ and $F^+$ be the infimum and supremum of the support of $F$, respectively:
  $$F^-=\inf\{\delta\ge 0|~\P(\tau_e<\delta)>0\},~F^+=\sup\{\delta\ge 0|~\P(\tau_e>\delta)>0\}.$$
\item We write $\log^{(2)}x=\log\log{x}$.
\item Given $a,b,y\in\R^d$, we define ${\rm T}(a,y,b)={\rm T}(a,y)+{\rm T}(y,b)$, which is the first passage time from $a$ to $b$ passing through $y$.
\item Given $\ell\in\N$, we define 
  \begin{align*}
    {\rm B}_{\infty}(\ell)&=[-\ell,\ell]^d\cap \Z^d.
    \end{align*}
\item Given a set $A\subset \Z^d$, we define the inner boundary $\overline{\partial} A$ as
  $$\overline{\partial} A=\{x\in A|~\exists y\notin A~\text{s.t.}~|x-y|_1=1\}.$$
\end{itemize}
\section{Proof}

Let $\mathbf{x}_d\in \partial \B_d$ be a directional flat point. Denote by $L$ a tangent plane of $\partial \B_d$ at $\mathbf{x}_d$. Remark that it is actually uniquely determined and $L$ is also the tangent plane of $\partial {\rm B}(x_1,r)$ at $\mathbf{x}_d$. Given sufficiently large $t>0$, one can find a finite subset $S_t$ of $t L$ such that the following hold:
\begin{equation}\label{Sn-cond}
  \begin{cases}
  \sharp S_t = [(\log{t})^{1/8}],\\
  \text{if $a\neq b\in S_t$, }|a-b|\ge t^{1/2}(\log{t})^{-1/8},\\
  \text{for any $a\in S_t$, }t^{1/2}(\log{t})^{-1/8}\le |a-t \mathbf{x}_{d}|\le t^{1/2}.\\
  \end{cases}
\end{equation}

 \noindent We state a basic property of a directional flat point.
      \begin{lem}\label{curved}
    Let $\B\subset \R^d$ be a convex subset and $\mathbf{x}_d\in
\partial \B$. Suppose that there exists $x_1\in\R^d$ and $r>0$ such
that ${\rm B}(x_1,r)\subset \B$ and $\mathbf{x}_d\in \partial {\rm B}(x_1,r)$. Let
$L$ be the unique tangent plane of $\partial {\rm B}(x_1,r)$ at
$\mathbf{x}_{d}$. Then there exists $C>0$ such that for any $t>1$ and
$y\in tL$ with $|y-t\mathbf{x}_d|\le \sqrt{t}$,
    $$d(y,\partial (t\B_d))\leq C.$$
  \end{lem}
\begin{proof}
By the rotation and translation, it suffices to prove the assertion in the case
where $d=2$, $x_1=r \mathbf{e}_2$ and $\mathbf{x}_d=0$ (See Figure~\ref{fig:NF}). Then $L=\{(x,0)|~x\in\R\}$. Note that $\partial (t{\rm B}(x_1,r))$ can be expressed locally as the graph of the function $x\rightarrow tr-t\sqrt{r^2-(x/t)^2}$
and if $|x|\le \sqrt{t}$, $tr-t\sqrt{r^2-(x/t)^2}\leq C$ with some
constant $C>0$ independent of $t$. Since $\partial (t\B_d)$ is located between
$tL$ and $\partial (t{\rm B}(tx_1,r))$, the desired bound $d(y,\partial (t\B_d))\leq C$
follows.
\end{proof}
By using Lemma~\ref{curved}, we get for any $y\in S_t$,
\begin{equation}
  |{\rm g}(y)-{\rm g}(t\mathbf{x}_d)|\leq 2dC\E \tau_e.\label{diff-g}
\end{equation}
We fix $\e\in(0,1/2)$ and $B>0$ to be a small constant and a large constant, respectively. If
$\sup_{y\in L_t\cap {\rm B}(t\mathbf{x}_d,\sqrt{t})}|\E[{\rm T}(0,y)]-{\rm g}(y)|> B t^{1/2-\e},$ then we can take such $[y]=x_n$ with $t=n$ to get Theorem~\ref{thm-main2}. In the following, we assume the contrary, i.e.,
\begin{align}\label{contrary}
  \sup_{y\in L_t\cap {\rm B}(t\mathbf{x}_d,\sqrt{t})}|\E[{\rm T}(0,y)]-{\rm g}(y)|\leq B t^{1/2-\e},
\end{align}
until \eqref{final-result}. Then by \eqref{diff-g}, we have
\begin{equation}\label{control-expect}
  \sup_{y\in L_t\cap {\rm B}(t\mathbf{x}_d,\sqrt{t})}|\E[{\rm T}(0,y)]-{\rm g}(t\mathbf{x}_d)|\leq 2B t^{1/2-\e}.
\end{equation}
 Note that for $y\in t L\cap {\rm B}(t\mathbf{x}_d,\sqrt{t})$, by shift invariance of the first passage time,
 \begin{align}\label{shif-inv}
   |\E[{\rm T}(y,2t\mathbf{x}_d)]-\E [{\rm T}(0,2t\mathbf{x}_d-y)]|\le 2d\E\tau_e,
   \end{align}
 and $2t\mathbf{x}_d-y\in tL\cap {\rm B}(t\mathbf{x}_d,t^{1/2}).$ Thus, 
\begin{equation}\label{control-expect2}
 \sup_{y\in L_t\cap {\rm B}(t\mathbf{x}_d,\sqrt{t})}|\E[{\rm T}(y,2t\mathbf{x}_d)]-{\rm g}(t\mathbf{x}_d)|\leq 2B t^{1/2-\e}+ 2d\E\tau_e \leq 3B t^{1/2-\e}.
\end{equation}
It is worth noting that \eqref{contrary} is used only in Lemma~\ref{estimate2} and the other arguments are free from this assumption.
We first estimate $\E[{\rm T}(0,t\mathbf{x}_d)] -{\rm g}(t\mathbf{x}_d)$ from below. The following observation, in particular \eqref{Key-ineq}, is simple but a powerful tool to get the lower bound of the non-random flucuation. In fact, we use a similar estimate to prove Theorem~\ref{lem-positive-NF}. 
\begin{prop}\label{estimate1}
Let $\mathcal{A}_y=\{\forall z\in S_t\text{ with $z\neq y$},~{\rm T}(0,y,2t\mathbf{x}_{d})< {\rm T}(0,z,2t\mathbf{x}_{d})\}$. For any $K>0$,
\begin{equation}
  2(\E {\rm T}(0,t\mathbf{x}_d)-{\rm g}(t\mathbf{x}_d))+4d\E \tau_e \geq K\sum_{y\in S_t}\P(\{{\rm T}(0,t\mathbf{x}_d,2t\mathbf{x}_d)-{\rm T}(0,y,2t\mathbf{x}_d)> K\}\cap\mathcal{A}_y).
\end{equation}
\end{prop}
We postpone the proof until the proofs of Theorem~\ref{thm-main2} and Corollary~\ref{thm-main} are completed. Let $M>0$ and $c=\frac{1}{32(1+M)}$. Then we take $K_t=(c\log^{(2)}(t))^{1/d}$. Next we will estimate $\P(\{{\rm T}(0,t\mathbf{x}_d,2t\mathbf{x}_d)-{\rm T}(0,y,2t\mathbf{x}_d)> K_t\}\cap\mathcal{A}_y)$ from below.
\begin{prop}\label{estimate2}
  If we take $M>0$ sufficiently large, then for any sufficiently large $t>1$ and $y\in S_t$,
 \begin{equation}
   \begin{split}
     &\quad \P(\{{\rm T}(0,t\mathbf{x}_d,2t\mathbf{x}_d)-{\rm T}(0,y,2t\mathbf{x}_d)> K_t\}\cap\mathcal{A}_y)\\
     &\geq \exp{(-MK_t^d)}(3/4-K_t^{-1}(\E[{\rm T}(0,y, 2t \mathbf{x}_{d})]-2{\rm g}(t\mathbf{x}_d)+4d\E\tau_e)). \label{form:est2}
     \end{split}
  \end{equation}
\end{prop}
 We prove our main theorems using the above propositions. We first suppose that there exists $y\in tL\cap {\rm B}(t\mathbf{x}_d,t^{1/2})$ such that  $\E [{\rm T}(0,y)]- {\rm g}(t\mathbf{x}_d)\geq K_t/8.$ By \eqref{diff-g},
$$\E[{\rm T}(0,y)] -{\rm g}(y)\geq \E [{\rm T}(0,y)] -{\rm g}(t\mathbf{x}_d)-2dC\E \tau_e\geq K_t/16.$$
    Otherwise, if for any $y\in tL\cap {\rm B}(t\mathbf{x}_d,t^{1/2})$, $\E [{\rm T}(0,y)- {\rm g}(t\mathbf{x}_d)]\leq K_t/8$, by \eqref{shif-inv}, then for sufficiently large $t>1$, we obtain
    \begin{equation}
      \begin{split}\E[{\rm T}(0,y,2t\mathbf{x}_d)]-2{\rm g}(t\mathbf{x}_d)+4d\E \tau_e&\leq K_t/4 +8d\E \tau_e\\
        &\leq K_t/2.
      \end{split}
      \end{equation}
      Recall that $\sharp S_t=[(\log{t})^{1/8}]$ and $K_t=(c\log^{(2)}(t))^{1/d}$. Combining with Proposition~\ref{estimate1} and \ref{estimate2},
\begin{equation}\label{final-result}
\begin{split}
\E[{\rm T}(0,t\mathbf{x}_d)] -{\rm g}(t\mathbf{x}_d)+2d\E\tau_e&\ge \frac{1}{8}K_t \sum_{y\in S_t} \exp{(-MK_t ^d)}\\
&= \frac{1}{8}K_t[(\log{t})^{1/8}] \exp{(-Mc\log^{(2)}{t})}> K_t/8.
\end{split}
\end{equation}    
Putting things together, with some constant $c>0$, we have that for sufficiently large $t>0$, there exists $y\in tL\cap {\rm B}(t\mathbf{x}_d,t^{1/2})$ such that $\E [{\rm T}(0,y)- {\rm g}(y)]\geq c(\log^{(2)}(t))^{1/d}$ under the assumption in Theorem~\ref{thm-main2}. This proves Theorem~\ref{thm-main2} by letting $[y]=x_n$ with $n=t$.\\

Next we will prove Corollary~\ref{thm-main}.
We write $D_t=\frac{c}{16d\E[\tau_e]}(\log^{(2)}(t))^{1/d}$. By \eqref{diff-g} and ${\rm g}(t\mathbf{x}_d)=t$, for any $z\in B\left(y,D_t\right)$, we get
\begin{equation}
  \begin{split}
    \E [{\rm T}(0,z)]-t&\geq  \E [{\rm T}(0,y)]-{\rm g}(y) -|{\rm g}(t\mathbf{x}_d)-{\rm g}(y)|-|\E [{\rm T}(0,z)]- \E [{\rm T}(0,y)]|\\
    &\geq \frac{c}{2}(\log^{(2)}(t))^{1/d},
    \end{split}
\end{equation}
which implies $z\notin G(t)$. Lemma~\ref{curved} yields that there exists $w\in B\left(y,D_t\right)$ such that $B\left(w,D_t/4\right)\subset t\mathbb{B}_d$. Then, since $d(w ,(t\B_d)^c)\geq D_t/4$, $w\in (t\B_d)_{D_t/4}^-$. Therefore, since $w\notin G(t)$ and $w\in \rmL(\mathbf{x }_d,r)$ for sufficiently large $t$, $(t\B_d)_{D_t/4}^-\cap \rmL(\mathbf{x }_d,r)\not\subset G(t)$, which implies
$${\rm F}_{\rmL(\mathbf{x }_d,r)}(G(t),t\B_d)\geq D_t/4.$$
\begin{proof}[Proof of Proposition\ref{estimate1}]
  For any $t>1$, observe that
  \begin{equation}\label{Key-ineq}
    \begin{split}
    &\quad 2(\E {\rm T}(0,t\mathbf{x}_d)-{\rm g}(t\mathbf{x}_d))+4d\E \tau_e\\
    &= \E [{\rm T}(0,t\mathbf{x}_d,2t\mathbf{x}_d)-{\rm T}(0,2t\mathbf{x}_d)]\\
    &\hspace{8mm}+(\E {\rm T}(0,t\mathbf{x}_d)-\E [{\rm T}(t\mathbf{x}_d,2t\mathbf{x}_d)]+2d\E \tau_e)+(\E[{\rm T}(0,2t\mathbf{x}_d)]-2{\rm g}(t\mathbf{x}_d)+2d\E \tau_e)\\
      &\geq \E [{\rm T}(0,t\mathbf{x}_d,2t\mathbf{x}_d)-{\rm T}(0,2t\mathbf{x}_d)],
      \end{split}
    \end{equation}
      where we have used  $\E {\rm T}(0,2t\mathbf{x}_d)+2d\E \tau_e\ge 2{\rm g}(t\mathbf{x}_d)$ and $|\E {\rm T}(0,t\mathbf{x}_d)-\E {\rm T}(t\mathbf{x}_d,2t\mathbf{x}_d)|\le 2d\E \tau_e$. \\
 Then since ${\rm T}(0,x,y)\ge {\rm T}(0,y)$ for any $x,y\in\R^d$ and $\{\mathcal{A}_y\}_{y\in S_t}$ are disjoint, we have
\begin{equation}
\begin{split}
\E [{\rm T}(0,t\mathbf{x}_d,2t\mathbf{x}_d)-{\rm T}(0,2t\mathbf{x}_d)]&\geq \sum_{y\in S_t}\E[{\rm T}(0,t\mathbf{x}_d,2t\mathbf{x}_d)-{\rm T}(0,2t\mathbf{x}_d);~\mathcal{A}_y]\\
&\ge \sum_{y\in S_t}\E[{\rm T}(0,t\mathbf{x}_d,2t\mathbf{x}_d)-{\rm T}(0,y,2t\mathbf{x}_d);~\mathcal{A}_y].
\end{split}
\end{equation}
 Then this is further bounded from below by 
\begin{equation}\begin{split}
 K\sum_{y\in S_t}\P(\{{\rm T}(0,t\mathbf{x}_d,2t\mathbf{x}_d)-{\rm T}(0,y,2t\mathbf{x}_d)> K\}\cap\mathcal{A}_y).
\end{split}\end{equation}
\end{proof}
We move to the proof of Proposition~\ref{estimate2}. We prepare some notations for the proof.
  \begin{Def}
 We define events $\mathcal{A}_1$ and $\mathcal{A}_2$ as
  \begin{equation}
    \begin{split}
\mathcal{A}_1&= \{\forall a,b\in {\rm B}(0,t^2)\text{ satisfying $|a-b|\ge t^{1/4}$,~${\rm T}(a,b)\ge (F^-+\delta)|a-b|_1$}\},\\
\mathcal{A}_2&= \{\forall y\in S_t,~\max_{z=0,2t\mathbf{x}_d}\{|{\rm T}(z,y)-\E[{\rm T}(z,y)]|\}\le t^{1/2}(\log{t})^{-1/4}\},
      \end{split}
  \end{equation}
  where $\delta$ will be defined in Lemma~\ref{useful} below.
   We set $\mathcal{A}=\mathcal{A}_1\cap \mathcal{A}_2$. 
\end{Def}
  \begin{Def}
Let $C$ be a positive constant to be chosen later.  
\begin{enumerate}
\item  A point $y\in S_t$ is said to be black if for any $a,b\in {\rm B}(y,C K_t)$ satisfying $|a-b|_1\ge K_t$, $${\rm T}(a,b)\ge (F^-+\delta)|a-b|_1. $$
\item A point $y\in S_t$ is said to be good if ${\rm T}(0,y, 2t \mathbf{x}_{d})-{\rm T}(0, 2t \mathbf{x}_{d}) < K_t$ and $y$ is black.
\end{enumerate}
\end{Def}
The following is a crucial property of a useful distribution.
\begin{lem}\label{useful}
If $F$ is useful, there exsit $\delta>0$ and $D>0$ such that for any $v,w\in\Z^d$,
$$\P({\rm T}(v,w)<(F^-+\delta)|v-w|_1)\leq{}e^{-D|v-w|_1}.$$
\end{lem}
For a proof of this lemma, see Lemma 5.5 in \cite{BK93}. As a consequence, we get
 \begin{equation}\label{A2}
    \lim_{t\to\infty}\inf_{S_t}\P(\mathcal{A}_1)=1\text{ and } \lim_{n\to\infty}\inf_{S_t}\min_{y\in S_t}\P(\text{$y$ is black})=1,
  \end{equation}
where $S_t$ runs over all subset of $t L$ satisfying \eqref{Sn-cond}. Moreover, we have the following.
\begin{lem}\label{sublinea-lem}
  \begin{equation}\label{A3}
    \lim_{t\to\infty}\inf_{S_t}\P(\mathcal{A}_2)=1.
\end{equation}
\end{lem}
\begin{proof}
  We use the sublinear variance \cite{BKS03,BR08,DHS15}: under the assumption $\E [\tau_e^2 (\log{\tau_e})_+] <\infty$, there exists $C>0$ depending only on $F$ and $d$ such that for any $x\in \R^d$,
  \begin{equation}
    {\rm Var}({\rm T}(0,x))\leq C\frac{|x|}{\log{(1+|x|)}}.
  \end{equation}
  Then by the union bound and Chebyshev's inequality, we have
  \begin{equation}
    \begin{split}
      &\P(\exists y\in S_{n}\text{ such that }\max_{z=0,2t\mathbf{x}_d}\{|{\rm T}(z,y)-\E[{\rm T}(z,y)]|\}\ge t^{1/2}(\log{t})^{-1/4})\\
      &\leq 2\sharp S_{n}\sup_{y\in  {\rm B}(t\mathbf{x}_d,t^{1/2})} \P(|{\rm T}(0,y)-\E[{\rm T}(0,y)]|\ge t^{1/2}(\log{t})^{-1/4})\\
      &\leq C'(\log{t})^{1/8}(\log{t})^{-1/2} \to 0,
      \end{split}
  \end{equation}
  where $C'$ is a positive constant depending only on $d$ and $F$.
  \end{proof}
\begin{lem}\label{key:lem}
 If we take $C>0$ sufficiently large depending on $\delta$, then for any sufficiently large $t>1$ and $y\in S_t$, the following holds:
  \begin{equation}
\begin{split}
&\quad\P(\{{\rm T}(0,t\mathbf{x}_d,2t\mathbf{x}_d)-{\rm T}(0,y,2t\mathbf{x}_d)> K_t\}\cap \mathcal{A}_y)\\
  &\ge \P(\forall e\subset {\rm B}(y,C K_t),~\tau_e\le F^-+\delta/2)\P(\mathcal{A}\cap\{\text{$y$ {\rm is good}}\}). \label{key} 
  \end{split}
\end{equation}
  \end{lem}
\begin{proof}
  We first explain the idea of the proof. We start with the event $\mathcal{A}\cap\{\text{$y$ is good}\}$. Then we resample all the configurations in ${\rm B}(y,CK_t)$ and consider the event that to each edge $e$ in ${\rm B}(y,CK_t)$, $\tau_e<F^-+\delta/2$ after resampling. Then it is easy to check that ${\rm T}(0,y,2t\mathbf{x}_d)$ decreases by at least $CK_t \delta/2$. On the other hand, since $y$ and $t\mathbf{x}_d$ are far away from each other, ${\rm T}(0,t\mathbf{x}_d,2t\mathbf{x}_d)$ is unchanged  or is much larger than ${\rm T}(0,y,2t\mathbf{x}_d)$ after resampling. Similarly, we have the same thing for $\{{\rm T}(0,z,2t\mathbf{x}_d)\}_{z\neq y}$. Thus we get $\{{\rm T}(0,t\mathbf{x}_d,2t\mathbf{x}_d)-{\rm T}(0,y,2t\mathbf{x}_d)> K_t\}\cap \mathcal{A}_y$ after resampling. To make the above heuristic rigorous, we use the resampling argument introduced in~\cite{BK93}.\\
  
 Let $\tau^*=\{\tau^*_e\}_{e\in \mathbb{E}^d}$ be an independent copy of $\{\tau_e\}_{e\in \mathbb{E}^d}$. We enlarge the probability space so that we can measure the event both for $\tau$ and $\tau^*$ and we still denote the joint probability measure by $\P$.   We define $\tilde{\tau}=\{\tilde{\tau}_e\}_{e\in \mathbb{E}^d}$ as $$\tilde{\tau}_e=\begin{cases}
    \tau_{e}^* & \text{if $e\subset {\rm B}(y,K_t)$}\\
  \tau_{e} & \text{otherwise.}
      \end{cases}$$
 We write $\tilde{\rm T}(a,b)$ for the first passage time from $a$ to $b$ with respect to $\tilde{\tau}$. We define $\tilde{\rm T}(a,y,b)$ similarly. Note that the distributions of $\tau$ and $\tilde{\tau}$ are the same under $\mathbb{P}$ since $\tau$ and $\tau^*$ are independent. Thus $\P(\mathcal{A}_y)=\P(\tilde{\mathcal{A}}_y)$, where $$\tilde{\mathcal{A}}_y=\{\forall z\in S_t\text{ with $z\neq y$},~\tilde{\rm T}(0,y,2t\mathbf{x}_{d})< \tilde{\rm T}(0,z,2t\mathbf{x}_{d})\}.$$
 Since the right hand side of \eqref{key} is equal to 
      \begin{equation}\label{key2}
        \P(\forall e\subset {\rm B}(y,CK_t),~\tilde{\tau}_e\le F^-+\delta/2,~\mathcal{A}\cap\{\text{$y$ is good}\})
      \end{equation}
      by independence of $\tau$ and $\tau^*$, it suffices to show that the event inside the probability in \eqref{key2} implies $\tilde{\mathcal{A}}_y$ and $\tilde{\rm T}(0,t\mathbf{x}_d,2t\mathbf{x}_d)-\tilde{\rm T}(0,y,2t\mathbf{x}_d)> K_t$. To do this, we suppose that $\tau$ and $\tilde{\tau}$ belong to the event in \eqref{key2}.\\

      \noindent \underline{Step 1} ($\tilde{\rm T}(0,y,2t\mathbf{x}_d)+ 2K_t<{\rm T}(0,y,2t\mathbf{x}_d)$)\\
      We take an arbitrary optimal path $\gamma=(\gamma_i)^l_{i=1}\subset \Z^d$ for ${\rm T}(0,y,2t\mathbf{x}_d)$. Let
      $$s=\min \{i\in\{1,\cdots,l\}|~\gamma_i\in {\rm B}(y,K_t)\}\text{ and }f=\max \{i\in\{1,\cdots,l\}|~\gamma_i\in {\rm B}(y,CK_t)\}.$$
      Since $\forall e\subset {\rm B}(y,CK_t),~\tilde{\tau}_e\le F^-+\delta/2$, we have
      $$\tilde{\rm T}(0,y,2t\mathbf{x}_d)\leq {\rm T}(0,\gamma_s)+{\rm T}(\gamma_f,2t\mathbf{x}_{d})+|\gamma_f-\gamma_s|_1(F^-+\delta/2).$$
      On the other hand, since $y$ is black and $\gamma$ passes through $[y]$, we have
      $${\rm T}(0,y,2t\mathbf{x}_d)\geq {\rm T}(0,\gamma_s)+{\rm T}(\gamma_f,2t\mathbf{x}_{d})+(|\gamma_f-\gamma_s|_1\lor |[y]-\gamma_s|_1)(F^-+\delta)$$
      Since $|\gamma_s-[y]|_1\ge CK_t-1$ and $C$ is sufficiently large depending on $\delta$,
      we have
      $$\tilde{\rm T}(0,y,2t\mathbf{x}_d)+ 2K_t<{\rm T}(0,y,2t\mathbf{x}_d).$$
      \underline{Step 2} ($\tilde{\rm T}(0,y,2t\mathbf{x}_d)+K_t< \tilde{\rm T}(0,z,2t\mathbf{x}_d)$ for any $z\in S_t\text{ with $z\neq y$}$ or $z=[t\mathbf{x}_d]$)\\
      Let $z\in S_t$ with $z\neq y$ or $z=[t\mathbf{x}_d]$. We first suppose that $\tilde{\rm T}(0,z,2t\mathbf{x}_d)< {\rm T}(0,z,2t\mathbf{x}_d)$. Then, since we resample the configurations only in ${\rm B}(y,CK_t)$, any optimal path $\gamma=(\gamma_i)^l_{i=1}$ for $\tilde{\rm T}(0,z,2t\mathbf{x}_d)$ must touch with ${\rm B}(y,CK_t)$, i.e., there exists $i\in\{1,\cdots,l\}$ such that $\gamma_i\in {\rm B}(y,CK_t)$.  By definition, $[z]$ is included in $\gamma$ and let $j\in\{1,\cdots,l\}$ be $\gamma_j=[z]$. We only consider the case $i<j$. In fact, for the other case, replacing \eqref{control-expect2} with \eqref{control-expect}, the same proof works. Then, by using the condition $\mathcal{A}_1$ and $\forall e\subset {\rm B}(y,CK_t),~\tilde{\tau}_e\le F^-+\delta/2$, respectively, we get
      \begin{align*}
        \tilde{\rm T}(0,y)&\leq \tilde{\rm T}(0,\gamma_i)+(F^-+\delta/2)CK_t,\\
        \tilde{\rm T}(0,z)&\geq \tilde{\rm T}(0,\gamma_i)+(F^-+\delta)|\gamma_i-[z]|_1,
        \end{align*}
      But, by $|y-z|\geq t^{1/2}(\log{t})^{-1/8}$,
      $$|\gamma_i-[z]|_1\geq t^{1/2}(\log{t})^{-1/8}-CK_t.$$
      Thus, 
      \begin{equation}\label{minus}
        \tilde{\rm T}(0,z)\geq \tilde{\rm T}(0,y)+\frac{1}{2}\left(F^-+\delta\right)t^{1/2}(\log{t})^{-1/8}.
      \end{equation}
      If there exists $i'>j$ such that $\gamma_{i'}\in {\rm B}(y,CK_t)$, as in \eqref{minus}, we have
      \begin{equation}
        \begin{split}
         \tilde{\rm T}(z,2t\mathbf{x}_d)\geq  \tilde{\rm T}(y,2t\mathbf{x}_d)+\frac{1}{2}\left(F^-+\delta\right)t^{1/2}(\log{t})^{-1/8}.
        \end{split}
      \end{equation}
      Otherwise, since we change the configurations only in ${\rm B}(y,CK_t)$, we have $\tilde{\rm T}(z,2t\mathbf{x}_d)\geq {\rm T}(z,2t\mathbf{x}_d)$. On the other hand, the essentially same argument as in Step 1 shows $\tilde{\rm T}(y,2t\mathbf{x}_d)+K_t< {\rm T}(y,2t\mathbf{x}_d)$. Combining with the condition $\mathcal{A}_2$ and \eqref{control-expect2} yields
      \begin{equation}
        \begin{split}
          &\quad \tilde{\rm T}(y,2t\mathbf{x}_d)- \tilde{\rm T}(z,2t\mathbf{x}_d)\\
          &\leq |{\rm T}(y,2t\mathbf{x}_d)-\E {\rm T}(y,2t\mathbf{x}_d)|+|\E {\rm T}(z,2t\mathbf{x}_d)-\E {\rm T}(y,2t\mathbf{x}_d)|+|\E {\rm T}(z,2t\mathbf{x}_d)-{\rm T}(z,2t\mathbf{x}_d)|\\
          &\leq 3t^{1/2}(\log{t})^{-1/4}.
        \end{split}
        \end{equation}
      In any case, together with \eqref{minus}, this gives
      \begin{align*}
        \tilde{\rm T}(0,y,2t\mathbf{x}_d)&\leq\tilde{\rm T}(0,z,2t\mathbf{x}_d)-\frac{1}{2}\left(F^-+\delta\right)t^{1/2}(\log{t})^{-1/8}+3t^{1/2}(\log{t})^{-1/4}\\
        &<\tilde{\rm T}(0,z,2t\mathbf{x}_d)-K_t.\end{align*}
      We now turn to the case $\tilde{\rm T}(0,z,2t\mathbf{x}_d)\geq {\rm T}(0,z,2t\mathbf{x}_d)$. Then, since  $y$ is good,
      $${\rm T}(0,y,2t\mathbf{x}_d)-K_t<{\rm T}(0,2t\mathbf{x}_d)\leq {\rm T}(0,z,2t\mathbf{x}_d),$$
      and thus Step 1 implies $\tilde{\rm T}(0,y,2t\mathbf{x}_d)+K_t<\tilde{\rm T}(0,z,2t\mathbf{x}_d)$. Thus the proof is completed.
  \end{proof}

\begin{proof}[Proof of Proposition~\ref{estimate2}]
  Since $\sharp\{e\subset {\rm B}(y,C K_t)\}\le 2d(2CK_t)^d$, by \eqref{A2} and \eqref{A3}, we will compute \eqref{key} as 
\begin{equation}
\begin{split}
&\quad \P(\forall e\subset {\rm B}(y,C K_t),~\tau_e\le F^-+\delta/2)\P(\mathcal{A}\cap\{\text{$y$ is good}\})\\
&\ge \P(\tau_e\le F^-+\delta/2)^{2d(2CK_t)^d}(\P(\{\text{$y$ is good}\})-\P(\mathcal{A}^c))\\
&\ge \exp\{(-MK_t^d)\}(\P({\rm T}(0,y, 2t \mathbf{x}_{d}) -{\rm T}(0, 2t \mathbf{x}_{d})< K_t)-1/4), \label{formula 3}
\end{split}
\end{equation}
with some constant $M>0$ independent of $t$. By using the first-moment method and $\E {\rm T}(0, 2t \mathbf{x}_{d})\geq  2{\rm g}(t\mathbf{x}_d)-4d\E \tau_e$, we get
\begin{equation}\begin{split}
\P({\rm T}(0,y, 2t \mathbf{x}_{d}) -{\rm T}(0, 2t \mathbf{x}_{d})< K_t)&\geq 1-K_t^{-1}\E[{\rm T}(0,y, 2t \mathbf{x}_{d}) -{\rm T}(0, 2t \mathbf{x}_{d})]\\
&\geq 1-K_t^{-1}(\E{\rm T}(0,y, 2t \mathbf{x}_{d}) -2{\rm g}(t\mathbf{x}_d)+4d\E\tau_e). \label{formula 2}
\end{split}\end{equation}
Therefore, the proof of Proposition~\ref{estimate2} is completed.
\end{proof}

\section{Proof of Theorem~\ref{lem-positive-NF}}
The following proof is almost independent of the previous arguments. Since $\E {\rm T}(-x,x)\geq {\rm g}(2x)$ for any $x\in\Z^d$,
\begin{equation}\label{Important-obs}
  2(\E {\rm T}(0,x)-{\rm g}(x))\geq  \E[{\rm T}(-x,0,x)-{\rm T}(-x,x)].
\end{equation}
Therefore, it remains to find an event independent of $x$, on which ${\rm T}(-x,0,x)-{\rm T}(-x,0,x)$ is uniformly bounded away from $0$. Let $F^+$ be the supremum of the support of the distribution $F$.\\

      \begin{figure}[b]
  \includegraphics[width=8.0cm]{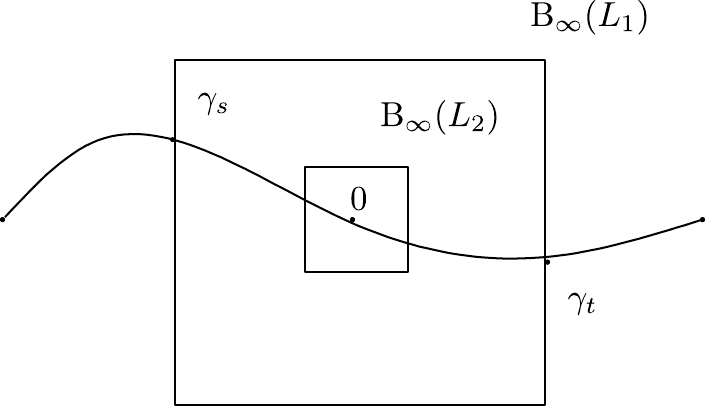}  \includegraphics[width=6.0cm]{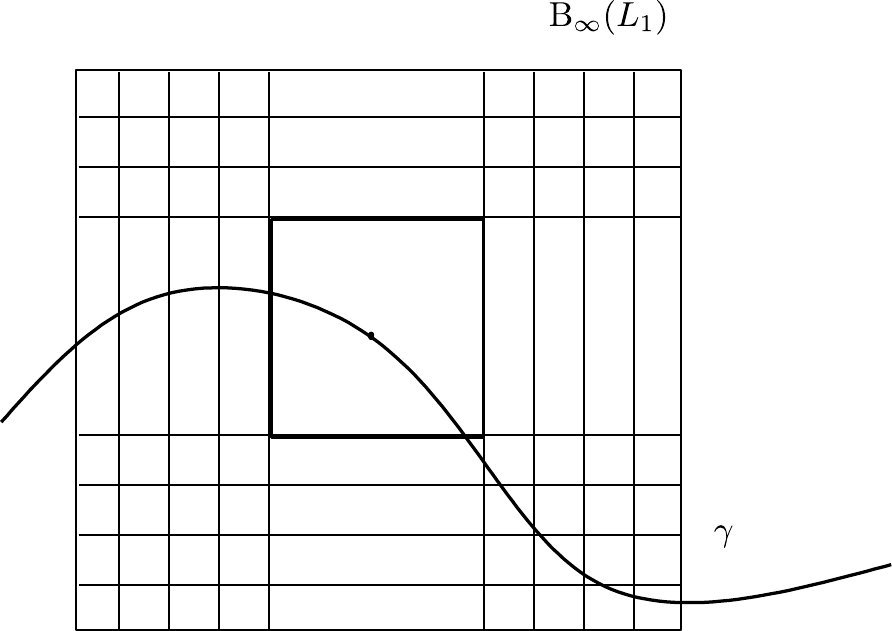}
  \caption{}
  \label{fig:NF2}
Left:  Figure of ${\rm B}_{\infty}(L_1)$ and ${\rm B}_{\infty}(L_2)$.\\
Right: Black lines represent $\tilde{\mathbb{L}}$.
      \end{figure}
      
We first consider the case $F^+=\infty$, which is rather easy. Let us define an event $A$ as
$$A=\{\forall e\in \mathbb{E}^d\text{ with }e\subset \overline{\partial} {\rm B}_{\infty}(1),~\tau_e\leq F^-+1\}\cap \{\forall e\in \mathbb{E}^d\text{ with }0\in e,~\tau_e\geq 2(F^-+2)\}.$$
      Then, on the event $A$, we have
      $${\rm T}(-x,0,x)-{\rm T}(-x,x)\geq 4(F^-+2)-4(F^-+1)\geq 4.$$
      Thus, \eqref{Important-obs} is further bounded from below by $4\P(A)$, which is uniformly bounded away from $0$.\\

      Next, we consider the general case. We take $\alpha\in(F^-,F^+ )$ arbitrary. Let $L_1>L_2\in\N$ be such that
      $$L_1\geq \left(\frac{F^-+1}{\alpha-F^-}\right) (L_2+1) \text{ and }L_2\geq \frac{1}{\alpha-F^-}.$$ Given $a\in  \overline{\partial} {\rm B}_{\infty}(L_1)$ and $i\in\{1,\cdots,d\}$, we define (see also Figure~\ref{fig:NF2})
      \begin{align*}
      {\rm L}_i(a)&=\{a+n\mathbf{e}_i|~n\in\Z\},\\
      \tilde{{\rm L}}_i(a)&=\{e\in \mathbb{E}^d|~e\subset {\rm L}_i(a)\cap {\rm B}_{\infty}(L_1) \},\\
      \tilde{\mathbb{L}}&=\underset{{\rm L}_i(a)\cap {\rm B}_{\infty}(L_2-1)=\emptyset}{\bigcup_{a\in  \overline{\partial} {\rm B}_{\infty}(L_1),~i\in\{1,\cdots,d\}}}\tilde{{\rm L}}_i(a).
      \end{align*}
      For the proof, we use the following proposition.
      \begin{prop}
       We take $\e=1/(16d L_1)$. Let us define an event $A$ as
$$A=\{\forall e\in \tilde{\mathbb{L}},~\tau_e\leq F^-+\e\}\cap \{\forall e\in \mathbb{E}^d \backslash \tilde{\mathbb{L}}\text{ with }e\subset {\rm B}_{\infty}(L_1),~\tau_e\geq \alpha\}.$$
        Then, on the event $A$, for any $x\in\Z^d\backslash {\rm B}_{\infty}(L_1)$,
      $${\rm T}(-x,0,x)-{\rm T}(-x,x)\geq 1.$$
      \end{prop}
      \begin{proof}
    We take any optimal path $\gamma=(\gamma_i)^{\ell}_{i=1}$ for ${\rm T}(-x,0,x)$ and define 
        $$s=\min\{i\in\{1,\cdots,\ell\}|~\gamma_i\in {\rm B}_{\infty}(L_1)\},~
        t=\max\{i\in \{1,\cdots,\ell\}|~\gamma_i\in {\rm B}_{\infty}(L_1)\}.$$
        We write $\gamma_{s,t}=(\gamma_i)^t_{i=s}$ and take $r\in\{s,\cdots,t\}$ such that $\gamma_r=0$.  Given $\ell,m\in\N$ verifying $m<\ell$, we define
        $$ {\rm B}_{\infty}(\ell,m)=\{x\in {\rm B}_{\infty}(\ell)|~\exists i\neq j\in\{1,\cdots,d\}~\text{s.t.}~|x_i|,|x_j|\geq m\}.$$

        (Case 1) First we suppose $\gamma_{s,t}\cap {\rm B}_{\infty}(L_1,L_2)=\emptyset$. Note that the graph distance between $\gamma_s$ and $\gamma_t$ in $\tilde{\mathbb{L}}$ is less than or equal to $|\gamma_s-\gamma_t|_1+2(L_2+1)$, which implies
        $${\rm T}(\gamma_s,\gamma_t)\leq (|\gamma_s-\gamma_t|_1+2(L_2+1))(F^-+\e)\leq (|\gamma_s-\gamma_t|_1+2(L_2+1))F^-+1.$$
        If $\gamma_{s,t}\cap {\rm B}_{\infty}(L_1,L_2)=\emptyset$, since $\gamma_{s,t}$ contains at least $2L_1$ edges of weights greater than $\alpha$, then ${\rm T}(\gamma_{s,t})\geq 2L_1\alpha+(|\gamma_s-\gamma_t|_1-2L_1) F^-$. Thus
        \begin{align*}
          {\rm T}(-x,0,x)&\geq {\rm T}(-x,\gamma_s)+2L_1\alpha+ (|\gamma_s-\gamma_t|_1-2L_1) F^-+{\rm T}(\gamma_t,x)\\
          &= {\rm T}(-x,\gamma_s)+2L_1(\alpha-F^-)+ |\gamma_s-\gamma_t|_1 F^-+{\rm T}(\gamma_t,x)\\
          &\geq {\rm T}(-x,\gamma_s)+2(L_2+1)F^-+2+ |\gamma_s-\gamma_t|_1 F^- +{\rm T}(\gamma_t,x)\\
          &\geq {\rm T}(-x,x)+1,
        \end{align*}
        where we have used $L_1\geq \left(\frac{F^-+1}{\alpha-F^-}\right)(L_2+1)$ in the third line.\\
        
        (Case 2) We suppose $\gamma_{s,t}\cap {\rm B}_{\infty}(L_1,L_2)\neq \emptyset$. By construction of ${\rm B}_{\infty}(L_1,L_2)$, it is straightforward to check that for any $y\in {\rm B}_{\infty}(L_1,L_2)$ and $z\in {\rm B}_{\infty}(L_1)$, the graph distance between $y$ and $z$ in $\tilde{\mathbb{L}}$ is $|y-z|_1$. In particular, we obtain ${\rm T}(y,z)\leq (F^-+\e)|y-z|_1.$ If $\gamma_{s,t}\cap {\rm B}_{\infty}(L_1,L_2)\neq \emptyset$, then we can take $i\in\{s,\cdots,t\}$ such that $\gamma_i\in {\rm B}_{\infty}(L_1,L_2)$. Without loss of generality, we can suppose $i< r$, since the other case can be treated in the same way. Since ${\rm T}(\gamma_i,\gamma_t)\leq |\gamma_i-\gamma_t|_1(F^-+\e)\leq |\gamma_i-\gamma_t|_1F^-+1$ and $\gamma_{i,t}=(\gamma_j)^t_{j=i}$ contains at least $2L_2$ edges of weights greater than $\alpha$,
        \begin{align*}
          {\rm T}(-x,0,x)&= {\rm T}(-x,\gamma_i)+{\rm T}(\gamma_{i},0,\gamma_t)+{\rm T}(\gamma_t,x)\\
          &\geq {\rm T}(-x,\gamma_i)+(|\gamma_i-\gamma_t|_1-2L_2)F^-+2L_2\alpha +{\rm T}(\gamma_t,x)\\
          &=   {\rm T}(-x,\gamma_i)+(|\gamma_i-\gamma_t|_1+2L_2(\alpha-F^-) +{\rm T}(\gamma_t,x)\\
          &\geq {\rm T}(-x,\gamma_i)+{\rm T}(\gamma_i,\gamma_t)+1+{\rm T}(\gamma_t,x) \geq {\rm T}(-x,x)+1,
        \end{align*}
        where we have used $L_2\geq (\alpha-F^-)^{-1}$ in the last line.
      \end{proof}
      If $x\notin {\rm B}_{\infty}(L_1)$, then
      \begin{align*}
        2(\E[{\rm T}(0,x)]-{\rm g}(x))&\geq \E[{\rm T}(-x,0,x)-{\rm T}(-x,x)]\\
        &\geq \P(A).
      \end{align*}
      Otherwise, if $x\in {\rm B}_{\infty}(L_1)$ with $x\neq 0$, then since $2L_1 x\notin {\rm B}_{\infty}(L_1)$, by using the sub-additivity of the first passage time, we get
      \begin{align*}
        2L_1(\E[{\rm T}(0,x)]-{\rm g}(x))&\geq \E[{\rm T}(0,2L_1 x)-{\rm g}(2L_1 x)]\\
        &\geq \frac{1}{2}\P(A).
      \end{align*}
      Thus, the proof is completed.


\section*{Acknowledgements}
The author would like to express his gratitude to Michael Damron, Syota Esaki and Ryoki Fukushima for helpful discussions and comments. Especially, thanks to the discussion with Ryoki Fukushima, the proof becomes simpler. He also thanks the anonymous referee for reading the paper carefully and providing thoughtful
comments. This research is partially supported by JSPS KAKENHI 16J04042.


\bibliographystyle{amsplain}




\end{document}